\documentclass[11pt]{amsart}
\usepackage[active]{srcltx}
\usepackage{amsmath,amssymb}
\usepackage{latexsym}
\usepackage{color}
\usepackage{dsfont}
\usepackage[normalem]{ulem}
\usepackage{epsfig}
\usepackage{pgf,pgfarrows,pgfnodes,pgfautomata,pgfheaps}
\newtheorem{teo}{Theorem}[section]
\newtheorem{lema}[teo]{Lemma}

\newtheorem{rem}{Remark}[section]

\usepackage{epstopdf}

\begin{document}

\title[Quasilinear heat equation with localized reaction]{Grow-up for a quasilinear heat equation with a localized reaction in higher dimensions}

\author{R. Ferreira and A. de Pablo}

\address{Ra\'{u}l Ferreira
\hfill\break\indent  Departamento de Matem\'{a}ticas,
\hfill\break\indent U.~Complutense de Madrid,
\hfill\break\indent
 28040 Madrid, Spain.
\hfill\break\indent  e-mail: {\tt raul$_-$ferreira@ucm.es}}

\address{Arturo de Pablo
\hfill\break\indent  Departamento de Matem\'{a}ticas,
\hfill\break\indent U. Carlos III de Madrid,
\hfill\break\indent
 28911 Legan\'{e}s, Spain.
\hfill\break\indent  e-mail: {\tt arturop@math.uc3m.es}}

\maketitle

\

\begin{abstract}
We study the behaviour of nonnegative  solutions to the quasilinear heat equation  with a reaction localized in a ball
$$
u_t=\Delta u^m+a(x)u^p,
$$
for $m>0$, $0<p\le\max\{1,m\}$, $a(x)=\mathds{1}_{B_L}(x)$, $0<L<\infty$ and $N\ge2$. We study when  solutions, which are global in time, are bounded or unbounded. In particular we show that the precise value of the length $L$ plays a crucial role in the critical case $p=m$ for  $N\ge3$. We also obtain the asymptotic
behaviour of unbounded solutions and prove that the grow-up rate is different in most of the cases to the one obtained when $L=\infty$.

Keywords: Quasilinear diffusion equations, localized reaction, grow-up.
\end{abstract}


\section{Introduction}

\label{sect-introduction} \setcounter{equation}{0}

We consider non-negative solutions to the following problem
\begin{equation}\label{eq.principal}
\left\{
\begin{array}{ll}
u_t(x,t)=\Delta u^m(x,t)+a(x) u^p(x,t),\qquad & (x,t)\in \mathbb R^N\times\mathbb R^+,\\
u(x,0)=u_0(x),
\end{array}\right.
\end{equation}
with $m,\,p>0$, $N\ge2$. We refer to~\cite{FerreiradePablo} for the case $N=1$. The initial datum is a continuous, nonnegative and nontrivial  function $u_0\in L^1(\mathbb{R}^N)\cap L^\infty(\mathbb{R}^N)$. The reaction coefficient is the characteristic function of the ball of radius $L$, $a(x)=\mathds{1}_{B_L}(x)$.

The existence of a solution to problem~\eqref{eq.principal}, local in time, can be easily achieved. To avoid uniqueness issues when $p<1$ we assume that $u_0$ is strictly positive in $\overline{B_{L}}$.  If $T$ is the maximal time of existence of the unique solution then the solution is bounded in $\mathbb{R}^N\times[0,t]$ for every $t<T$.

Problems like \eqref{eq.principal} are studied mainly when $p>1$ and in the context of blow-up, i.e. when  $T$ is finite, and in that case
\begin{equation}\label{bup}
\|u(\cdot,t)\|_\infty\to \infty \quad \mbox{as } t\to T.
\end{equation}
The case with global reaction, $L=\infty$, has been described
by Fujita in the semilinear case $m=1$ in the the seminal work \cite{Fujita}. It is proved in that paper that there exist two exponents, the \emph{global existence exponent} $p_0=1$ and the so called \emph{Fujita exponent} $p_F=1+2/N$, such that for $0<p< p_0$ all the solutions are globally defined in time, for $p_0<p< p_F$ all the solutions blow up, whereas for $p>p_F$ there exist both, global solutions and blowing-up solutions. The limit cases $p=p_0$ and $p=p_F$ belong, for this problem and respectively, to the global existence range and blow-up range, see also \cite{Hayakawa}. From this result several extensions have been investigated in the subsequent years, for all values of $m>0$, or with different diffusion operators and reactions; we mention the monographs \cite{GalaktionovKurdyumovMikhailovSamarski,QuittnerSouplet} for equation \eqref{eq.principal} with $L=\infty$.

In the presence of a localized reaction, $L<\infty$, problem \eqref{eq.principal} has been studied, again  in the context of blow-up, in \cite{BaiZhouZheng,FerreiradePabloVazquez,Liang,Pinsky}.
\begin{teo}\label{teo-exponents} The global existence exponent and the Fujita exponent for problem~\eqref{eq.principal} are
  \begin{equation}\label{exp-p0}
  \begin{array}{ll}
p_0=\max\{1,\dfrac{m+1}2\},\quad p_F=m+1,&\quad\text{if } N=1, \\ [3mm]
p_0=p_F=\max\{1,m\},&\quad\text{if } N\ge2.
\end{array}
\end{equation}
\end{teo}
We remark that the fast diffusion case $m<1$  is not covered by those references, but the result is trivial in that range, see Lemma~\ref{lem-exp}.

Our purpose in this work is to study the behaviour of global solutions in the lower interval $p\le p_0$, and to characterize wether they are bounded or not. In the last case we have that~\eqref{bup} holds with $T=\infty$, a phenomenon that is called {\it grow-up}. When $L=\infty$ every solution with $p\le p_0=1$ has grow-up, see \cite{dePabloVazquez}. As Theorems~\ref{teo-global-p0} and~\ref{teo-GUP} below show, the situation where $L<\infty$ is much more involved. This is in particular true in the  case $p=m$ and $N\ge3$, where there exists a critical length $L^*=L^*(N)$, see \eqref{stationary-bessel},   delimiting two very different behaviours.

In the critical exponent case $p=p_0$ the solutions can be global or not, contrary to what happens when $L=\infty$. In fact $p=p_0$ lies in the global existence side if $N=1$, see \cite{FerreiradePabloVazquez,BaiZhouZheng} or $N\ge2$ and $m\le1$, see Lemma~\ref{lem-exp}. We emphasize that the case $p=m>1$, $N\ge3$, is  not clear in the bibliography, see \cite[Theorem 1.1]{Liang} where the  author asserts that the solution blows up for every $L>0$, which is not true when $L$ is small.
\begin{teo}
  \label{teo-global-p0} Let $p=p_0$. The solution to problem~\eqref{eq.principal} is always globally defined in time if $p_0=1$, while if $p_0=m>1$ it is global if and only if  $0<L\le L^*$, $N\ge3$.
\end{teo}

We now study the global solutions in the range $p\le p_0$. The behaviour depends on the dimension and also  on $p$, $m$ and $L$.
Comparison with stationary solutions or unbounded explicit solutions will be useful, see Section~\ref{sect-special}.
We include here the case $N=1$ studied in~\cite{FerreiradePablo} for completeness.

\begin{teo}\label{teo-GUP} Let $p\le p_0$ and let $u$ be a global solution to problem~\eqref{eq.principal}.
\begin{itemize}
  \item If $N=1$ then $u$ is unbounded.
  \item If $N=2$ and the initial value is large then $u$ is unbounded. If in addition $p\le m$ then $u$ is always unbounded.
  \item Let $N\ge3$.
  \begin{itemize}
    \item If  $p< m$ then $u$ is bounded.
    \item If $p=m$ then $u$ is bounded when $L\le L^*$ (assuming also \eqref{extrainfinity} if $L=L^*$), and  $u$ is  unbounded when $L>L^*$.
    \item If $p>m$ then $u$ can be bounded or unbounded depending on the initial value.
  \end{itemize}
\end{itemize}
\end{teo}

\begin{figure}[h]
\hspace*{-1.2cm}
\includegraphics[width=14cm]{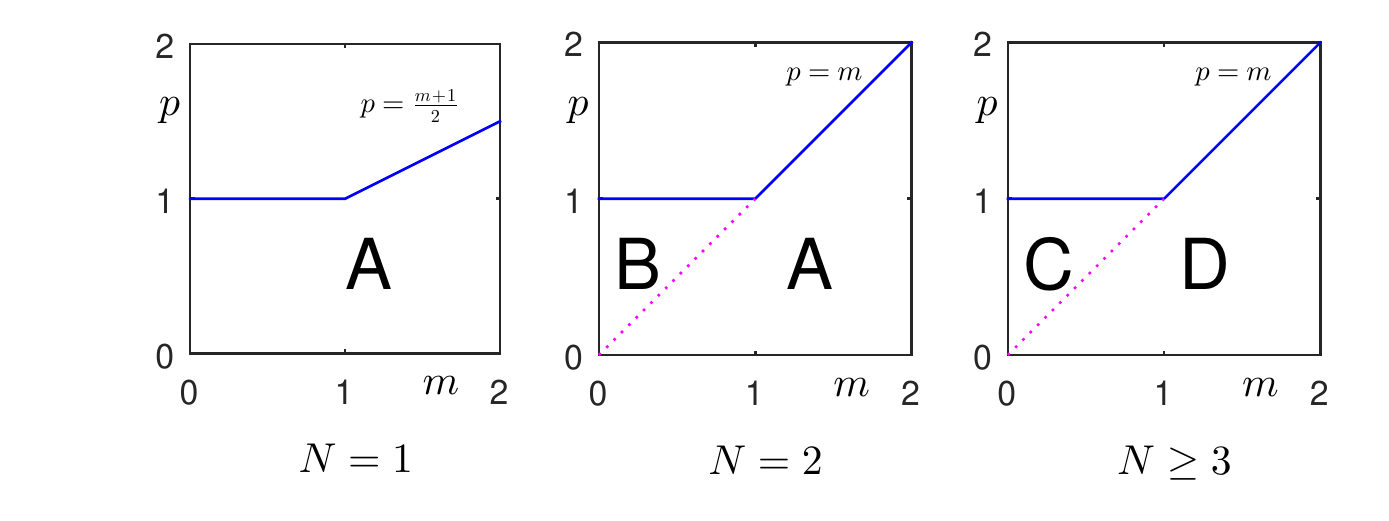}
\caption{Grow-up regions below $p_0$ in dimensions $N=1$, $N=2$ and $N\ge3$ resp.: $A$, all solutions are unbounded; $B$, there exist unbounded solutions; $C$, there exist bounded and unbounded solutions; $D$, all solutions are bounded.}
\label{fig.gup-vs-bdd}
\end{figure}

Each border line belongs to the corresponding subset below it, except for the line $p=m\le1$ when $N\ge3$, where the boundedness of the solutions depends on the length $L$. In the critical case $p=m$ and $L=L^*$ when $N\ge3$ we must impose an extra condition on the behaviour of the initial value at infinity, namely
\begin{equation}
  \label{extrainfinity}
  \limsup_{|x|\to\infty}|x|^{\frac{N-2}m}u_0(x)<\infty.
\end{equation}
Under this condition the solution is always bounded.

We do not know if all the solutions grow up in the parameter range $m<p\le1$ when $N=2$, that is if $B$ is actually $A$ or $C$ in Fig.~\ref{fig.gup-vs-bdd}. We  notice that in the linear reaction case $p=1$ with superfast diffusion $0<m<m_*\equiv \frac{N-2}N$, $N\ge3$, besides unbounded solutions there also exist solutions that vanish identically in finite time, see Remark~\ref{rem-vanish}.

For the unbounded solutions to problem \eqref{eq.principal} we also characterize the grow-up set. We assume for simplicity that the initial datum is radial. This is not a restriction  if $p\le m$ with $N=2$ or if $p=m<1$ with $L>L^*$ when $N\ge3$, since then every solution is unbounded and we may use comparison with a smaller initial datum satisfying those properties. When $m<p\le 1$ we obtain the grow-up set only for radial unbounded solutions with limit infinity at some pint.
\begin{teo}\label{teo-sets}
  Let $u$ be a global unbounded solution to problem~\eqref{eq.principal}. If $m<p\le 1$ assume also that $u$ is radial satisfying $\lim\limits_{t\to\infty}u(x,t)=\infty$ for $|x|=R$ and some $R\ge0$. Then $\lim\limits_{t\to\infty}u(x,t)=\infty$ uniformly in compact sets.
\end{teo}
We do not know if in the upper range $m<p\le 1$ there exist unbounded solutions with $\liminf\limits_{t\to\infty}\|u(\cdot,t)\|_\infty<\infty$.

Once the existence of global unbounded solutions is characterized, the main question to deal with  is to determine the grow-up rate, that is, the speed at which they go to infinity.
An easy upper estimate of the grow-up rate for $0<p\le1$ is given by the solutions of the ODE $U'(t)=U^p(t).$
This gives,
\begin{equation}
  \label{flat0}
  u(x,t)\le\begin{cases}  ct^{\frac1{1-p}},& \quad \text{if } p<1, \\
  ce^t,& \quad \text{if } p=1.
  \end{cases}
\end{equation}
We call this the \emph{natural rate}.

If $L=\infty$ and $0<p< p_0=1$ the grow-up rate is indeed given by the natural rate, that is,
$$
u(x,t)\sim t^{\frac1{1-p}},
$$
where by the symbol $f\sim g$ we mean $0<c_1\le f/g\le c_2<\infty$. If $p=1$ we can perform a change of variables to eliminate the reaction term as in \cite{FerreiradePablo}, getting in this way an exponential grow-up,
$$
u(x,t)\sim\begin{cases}
t^{-\frac N2}e^t,& \quad \text{if } m=1, \\
e^{\frac{2t}{N(m-1)+2}},& \quad \text{if } m>1, \\
e^t,& \quad \text{if } m<1.
\end{cases}
$$
We write this in weak form as
\begin{equation}\label{log-p=1}
  \lim\limits_{t\to\infty}\frac{\log u(x,t)}t=\min\{1,\frac2{N(m-1)+2}\}.
\end{equation}
We remark that this change is not possible if $L<\infty$.

We prove in this paper that for the case of a localized reaction estimates \eqref{flat0} are not always sharp, that is,
the grow-up rate for problem~\eqref{eq.principal} is in most of the cases strictly less than the natural grow-up rate. Let us see this phenomenon heuristically. If we perform the  rescaling, for $p\ne 1$,
\begin{equation}\label{rescale}
v(\xi,\tau)=t^{-\alpha} u(x,t),\qquad \xi=xt^{-\beta},\;\tau=\log t,
\end{equation}
with $\alpha(m-1)-2\beta+1=0$, we have that $v$ is a solution to the equation
\begin{equation}\label{rescaled-eq}
v_\tau=\Delta v^m+\beta\xi\nabla v+b(\xi,\tau)v^p-\alpha v,\qquad
\end{equation}
where the reaction coefficient becomes
\begin{equation}\label{rescaled-reaction}
b(\xi,\tau)=e^{\gamma\tau}\mathds{1}_{\{|\xi|<Le^{-\beta\tau}\}}(\xi),\qquad \gamma=\alpha(p-1)+1\ge 0.
\end{equation}
Now choose  the natural rate $\alpha=1/(1-p)$. This implies $\beta=(m-p)\alpha/2$ and $\gamma=0$. Therefore, when $p>m$ we have $\beta<0$ and the reaction coefficient tends to 1 in the whole $\mathbb{R}^N$, so the rescaled solution $v$ is supposed to stabilize to the constant $ (1-p)^{\frac1{1-p}}$, at least below the critical Sobolev exponent $p_S=m(N+2)/(N-2)$ . The grow-up rate must then be the natural one. The proof of this fact is our first result. The importance of the critical Sobolev exponent $p_S$  is well known in the characterization of the blow-up rates in superlinear problems, see for instance \cite{QuittnerSouplet}.

\begin{teo}
  \label{teo-rates-N>1}
Let $m<p< \min\{p_S,1\}$.
If $u$ is a solution to problem~\eqref{eq.principal} with global grow-up then as $t\to\infty$,
\begin{equation*}
    \label{rates11}
    u(x,t)\sim t^{\frac1{1-p}}
\end{equation*}
for every $|x|<L$.
\end{teo}

If $p=m$, the parameters $\beta$ and $\gamma$ are zero and the rescaled solution $v$ is supposed to stabilize to a (nonconstant) positive stationary profile. This suggest again that the grow-up rate of $u$ is the natural one. But the fact that $x=\xi$ in that case gives that the rate must hold in the whole space.
\begin{teo}
Let $p=m<1$ and let $u$ be a global unbounded solution of \eqref{eq.principal}. Then
$$
u(x,t)\sim t^{\frac1{1-m}}
$$
uniformly in compact sets of $\mathbb{R}^N$.
\end{teo}

When $p<m$ (which implies $N=2$ in order to have grow-up) the reaction coefficient disappears in the limit and the function $v$ must tend to zero. This gives that the rate should be strictly smaller than the natural one. Following  what is done in \cite{FerreiradePablo} to treat the one dimensional problem, we look at the case when the reaction coefficient tends to a Dirac delta at the origin, which for $N=2$ implies $\gamma=2\beta$. This means $\alpha=0$, which suggests a logarithmic grow-up rate. Thanks to Duhamel's formula we prove that this is indeed what happens in the case of linear diffusion $m=1$.

\begin{teo}
  \label{teo-rates-p<1}
Assume $N=2$, $m=1$, $0<p<1$ and let $u$ be a solution to problem~\eqref{eq.principal}. Then
\begin{equation}
  \label{rates-p<1}
 u(x,t)\sim  (\log t)^{\frac1{1-p}}
\end{equation}
uniformly  in compact sets.
\end{teo}

The case $p < m\ne 1$ in dimension $N = 2$ will be the subject of a separate
work.

Finally when $p=1\ge m$ we cannot perform the previous rescaling and we must try an exponential type change of variables. The argument therefore  suggests an exponential grow-up, $\log\|u(\cdot,t)\|_\infty\sim \lambda t$. The main point is that the natural rate $\lambda=1$ is obtained only if $m<1$, whereas when $m=1$ the rate is smaller and it depends on the length $L$.
\begin{teo}
  \label{teo-rates-N>1}
Let  $p=1$ and $\frac{N-2}{N+2}<m\le1$. If $u$ is a solution to problem~\eqref{eq.principal} with global grow-up then  for $t\to\infty$,
\begin{enumerate}
\item if  $m<1$
  \begin{equation*}
    \label{rates21}
      u(x,t)\sim e^t
  \end{equation*}
  for every $|x|<L$;
\item if $m=1$ and $L>L^*$ there exists a  function $\lambda_0=\lambda_0(L)\in(0,1)$ such that
\begin{equation}
  \label{rates-p=1}
  \lim\limits_{t\to\infty}\frac{\log u(\cdot,t)}t=\lambda_0
\end{equation}
uniformly in compact sets.
\end{enumerate}
\end{teo}
The function $\lambda_0(L)$ is increasing in $(L^*,\infty)$ and satisfies $\lim\limits_{L\to L^*}\lambda_0(L)=0$, $\lim\limits_{L\to\infty}\lambda_0(L)=1$, see~\eqref{lambda-L}. Observe also the influence again of the Sobolev exponent $p_S$.

Let us note that for $m<p\le1$ we  have obtained the grow-up rate only inside $B_L$. Our last result involves the characterization of  the grow-up rate outside $B_L$ and prove that, under certain restrictions, it is different (smaller) from the rate inside  the ball. This is particularly outstanding in the case $p=1$, when the solution grows  outside like a power, which is  much slower than the exponential growth inside. The proof uses comparison with solutions of the pure diffusion equation of a particular self-similar form, see Section~\ref{sect-special}. The existence of such special solutions will require to consider (not too) fast diffusion, $m^*<m<1$.
\begin{teo}\label{teo.tasas.fuera}
Let $m<p\le 1$ with $m>m^*$ and $p<p_S$ if $N\ge3$. Assume that there exists $C>0$ such that for $|x|$ large
$$
u(x,0)
\le C |x|^{\frac{-2}{1-m}}
\qquad \mbox{if } p<1
$$
or
$$
u(x,0)\sim |x|^{\frac{-2}{1-m}}(\log(x)^{\frac1{1-m}} \qquad \mbox{if } p=1.
$$
Then, for every $|x|>L$ it holds
$$
u(x,t)\sim t^{\frac1{1-m}}.
$$
\end{teo}

The paper is organized as follows: Section~\ref{sect-special} is devoted to the existence of special solution: stationary solutions, exponential unbounded solutions for the linear equation, and self-similar solutions to the pure fast diffusion equation; Section~\ref{sect-bdd-notbb} deals with the question of whether the global solutions below the global existence exponent $p_0$ are bounded or not; in Section~\ref{sect-gupset} we show that the grow-up set is $\mathbb{R}^N$ generically; finally in Section~\ref{sect-guprate} we study the rate at which the unbounded solutions tend to infinity.

\section{Special solutions}\label{sect-special}

In this Section we study three families of special solutions, namely stationary solutions, explicit unbounded solutions and self-similar solutions. We first characterize the existence of stationary solutions, both for the Cauchy problem and for the corresponding Dirichlet problem in a ball. We then study the existence of explicit unbounded solutions with exponential growth in the linear equation.  We finally construct certain type of self-similar solutions for the pure fast diffusion equation.
\subsection{Stationary solutions}\label{sect-stationary}

We show here that problem \eqref{eq.principal} admits stationary solutions for every $N\ge3$ and any $p>0$, independent of the Sobolev exponent $p_S$. We concentrate in radial solutions, $u=u(r)$, $r=|x|$. We also consider later  the corresponding Dirichlet problem in a ball. We remark that in this last case the critical Sobolev exponent does play a role.
\begin{teo}\label{teo-stat}
  Problem \eqref{eq.principal} possesses positive radial stationary solutions only if $N\ge3$. Moreover
  \begin{enumerate}
    \item If $p<m$, for any $k\ge0$ there exist a unique  stationary solution such that $\lim\limits_{r\to\infty}u(r)=k$.
    \item If $p>m$, there exists a finite value $k^*>0$, such that there exist   stationary solutions with  $\lim\limits_{r\to\infty}u(r)=k$ if and only if $0\le k\le k^*$  if $p<p_S=\frac{m(N+2)}{N-2}$, or $0< k\le k^*$  if $p\ge p_S$.
        \item If $p=m$, there exists a critical length $L^*=L^*(N)$, such that there exist stationary solutions if and only if $L\le L^*$. The solution is  characterized by $k=\lim\limits_{r\to\infty}u(r)$, and is unique for any $k>0$  if $L<L^*$, while it is unique up to a multiplicative constant if $L=L^*$, in which case $k=0$.
  \end{enumerate}
\end{teo}
\begin{proof}
Putting  $w=u^m$ we obtain $w$ matching two functions  for $r<L$ and $r>L$, respectively. More precisely, $w$ is given by
\begin{equation}\label{stat-w}
w(r)=
\begin{cases}
Av(A^{\frac{\gamma-1}2}r),\qquad\text{for } 0<r<L, \\
c_1+c_2\phi(r),\qquad\text{for } r\ge L,
\end{cases}
\end{equation}
where $\gamma=p/m$, $\phi$ is the Green function
\begin{equation}
  \label{green}
\phi(r)=\begin{cases}
  r^{2-N},&\text{ if } N\ne2, \\ \log r,&\text{ if } N=2,
\end{cases}
\end{equation}
and $v$  satisfies
\begin{equation}
  \label{stat-rad1}
\begin{cases}
  v''+\dfrac{N-1}rv'+v^\gamma=0, \\
  v(0)=1,\quad v'(0)=0.
\end{cases}
\end{equation}

We observe that there exist no nonnegative stationary solution if $N=1$ or $N=2$ since the Green function is unbounded in those dimensions. Let then be $N\ge3$. It is well known that there exists a unique function $v$ solution to \eqref{stat-rad1} defined in a maximal interval $[0,r_0)$, which is positive and decreasing in $0<r<r_0$ and $\lim\limits_{r\to r_0}v(r)=0$, where $r_0<\infty$ if $0<\gamma<\gamma_S=\frac{N+2}{N-2}$, while $r_0=\infty$ if $\gamma\ge \gamma_S$. Moreover $v$ is explicit in the limit case $\gamma=\gamma_S$, namely $v(r)=(1+Br^2)^{\frac{2-N}2}$, $B=\frac1{N(N-2)}$, while $\lim\limits_{r\to\infty}r^{\frac2{\gamma-1}}v(r)=K(\gamma,N)$ if $\gamma>\gamma_S$.  See for instance~\cite[Lemma~3.IV.1]{GalaktionovKurdyumovMikhailovSamarski}.

In order to match the two pieces at $r=L$ we have to choose $A>0$ properly.
First we must have $A^{\frac{\gamma-1}2}L<r_0$ when $r_0$ is finite. Thus, depending on the sign of the exponent, we see that $v(A^{\frac{\gamma-1}2}L)>0$ for $A$ large if $\gamma<1$, for $A$ small if $1<\gamma<\gamma_S$,  and for $L<L_1$ if $\gamma=1$, where $L_1$ is the radius of the ball for which the first eigenvalue of the Laplacian is 1. We then study the matching conditions,
\begin{equation}
  \label{matching}
  \begin{cases}
    c_1+c_2L^{2-N}=Av(A^{\frac{\gamma-1}2}L), \\
    (N-2)c_2L^{1-N}=-A^{\frac{\gamma+1}2}v'(A^{\frac{\gamma-1}2}L),
  \end{cases}
\end{equation}
and characterize when it is $c_1\ge0$. Observe that $c_2>0$ trivially under the above conditions on $A$ or $L$. We get
$$
c_1=Av(A^{\frac{\gamma-1}2}L)+\frac L{N-2}A^{\frac{\gamma+1}2}v'(A^{\frac{\gamma-1}2}L)= AF(A^{\frac{\gamma-1}2}L),
$$
where
$$
F(r)=v(r)+\frac1{N-2}rv'(r).
$$
We compute
$$
\begin{array}{l}
F'(r)=\dfrac1{N-2}\big((N-1)v'(r)+rv''(r)\Big)=-\dfrac{rv^\gamma(r)}{N-2}<0 \quad \text{for } 0<r<r_0, \\ [3mm]
F(0)=1>0,\quad \begin{cases}F(r_0)=\dfrac{r_0v'(r_0)}{N-2}<0,&\text{if } \gamma<\gamma_S, \\
\lim\limits_{r\to\infty}F(r)=0,&\text{if } \gamma\ge\gamma_S.
\end{cases}
\end{array}
$$
The precise behaviour of $F$ at infinity in this latter cases is
$$
\left\{\begin{array}{ll}
F(r)=(1+Br^2)^{-\frac N2}, &\mbox{if } \gamma=\gamma_S, \\
F(r)\sim r^{-\frac 2{\gamma-1}}, &\mbox{if } \gamma>\gamma_S.
\end{array}\right.
$$
Then, if $\gamma<\gamma_S$ we have that $F$ has a unique root $0<r^*<r_0$, while $F(r)>0$ for every $r>0$ if $\gamma\ge\gamma_S$. Let $A^*=(r^*/L)^{\frac2{\gamma-1}}$ if $\gamma<\gamma_S$, $\gamma\neq1$. We thus have:

- If $\gamma<1$, the function $c_1=c_1(A)$ is positive and increasing in $A\in(A^*,\infty)$,  vanishes at $A=A^*$ and satisfies $\lim\limits_{A\to\infty}c_1(A)=\infty$.

- If $1<\gamma<\gamma_S$, then $c_1(A)$ is positive  in $A\in(0,A^*)$, and vanishes at $A=0$ and $A=A^*$.

- If $\gamma=\gamma_S$ it is $c_1(A)>0$ for every $A>0$ and $\lim\limits_{A\to\infty}c_1(A)=0$.

- If $\gamma>\gamma_S$ it is again $c_1(A)>0$ for every $A>0$, but $\lim\limits_{A\to\infty}c_1(A)>0$.

- If $\gamma=1$ we have $c_1(A)=AF(L)>0$ for every $A>0$ provided $0<L< L^*\equiv r^*$, $c_1(A)=0$ when $L=L^*$.

\

We have  characterized the existence of the stationary solutions in terms of the value $A=w(0)$, if $\gamma\ne1$, and in terms of the length $L$ when $\gamma=1$. The limit at infinity is
$k=\lim\limits_{r \to\infty} w(r)=c_1(A)$. In summary we have obtained that, if $\gamma<1$, for each $0\le k<\infty$ there exists a unique $A=c_1^{-1}(k)\in[A^*,\infty)$, whereas for $\gamma>1$ there exists some $A$ with $c_1(A)=k$ when $0\le k\le k^*$ if $\gamma<\gamma_S$ or when $0< k\le k^*$ if $\gamma\ge\gamma_S$; the maximum value $k^*$  is given by
$$
k^*=\max\limits_{0<r<A^*}c_1(A),
$$
where we put $A^*=\infty$ when $\gamma\ge \gamma_S$.
\end{proof}

\begin{rem}\label{rem-chino}
From this result we immediately deduce that there exist bounded solutions to problem \eqref{eq.principal} for $N\ge3$ and $p=m$ provided $L$ is small, which contradicts \cite{Liang}.
\end{rem}

The stationary solutions are explicit in the case $p=m$ since $v$ can be written in terms of Bessel functions $v(r)=r^{\frac{2-N}2}J_{\frac{N-2}2}(r)$. The matching condition is
    \begin{equation}\label{stationary-bessel}
    \frac{N-2}2J_{\frac{N-2}2}(L)+L J'_{\frac{N-2}2}(L)=0,
    \end{equation}
and $L^*$ is the first positive root of that equation.

For instance when $N=3$ we have $L^*=\pi/2$, and for $L\le \pi/2$ the solution is any multiple of
  \begin{equation}\label{stationary-seno}
    w(r)=\left\{\begin{array}{ll}
    \dfrac1r\,\sin r,&\text{ for } 0<r<L, \\ [3mm] \dfrac1r(\sin L-L\cos L)+\cos L,&\text{ for } r\ge L.
  \end{array}\right.
  \end{equation}

As a byproduct of the above calculations, just looking at negative values of the function $c_1(A)$, we describe the existence of stationary solutions for the Dirichlet problem. That is, we consider,  for some $R>L$, the problem
\begin{equation}\label{eq.dirichlet}
\left\{
\begin{array}{ll}
\Delta w+a(x) w^\gamma=0,\qquad & |x|<R,\\
w(x)>0,\qquad & |x|<R,\\
w(x)=0,\qquad & |x|=R.
\end{array}\right.
\end{equation}
Clearly if $0<R\le L$ and $\gamma<\gamma_S$, $\gamma\neq1$, the solution is given by $w(x)=Av(A^{\frac{\gamma-1}2}|x|)$, $A=\left(\frac{R}{r_0}\right)^{-\frac2{\gamma-1}}$, where $v$ is the solution to \eqref{stat-rad1}, while if $\gamma=1$ the solutions (any multiple of $v$) exist only when $R=L_1$, the length for which the eigenvalue of the Laplacian is 1. We consider here all dimensions $N\ge1$ (we set $L^*(N)=0$ for $N\le2$).

\begin{teo}\label{estacionarias-Dirichlet}
Problem \eqref{eq.dirichlet} with $R>L$ possesses bounded positive solutions  if $N\le2$ or  if $N\ge3$ and $\gamma<\gamma_S=\frac{N+2}{N-2}$. They are radially decreasing. Moreover,
 \begin{enumerate}
   \item If $0<\gamma<1$ they exist for every $R>L$. The value at the origin $w(0)=A=A(R)$ increases with $R$.
   \item If $1<\gamma<\gamma_S$ they exist for every $R>L$. The value $A=A(R)$ decreases with $R$.
   \item If $\gamma=1$ they exist only if $L^*<L<L_1$ and only for a precise value $R=R(L)>L_1$, which is decreasing in $L$.
 \end{enumerate}
\end{teo}

In the case $0<\gamma<\gamma_S$, $\gamma\neq1$, it is easy to establish the asymptotics, for $R\to\infty$
\begin{equation}
  \label{eq.A-R}
\begin{array}{ll}
  A^{1-\gamma}\sim\dfrac R{L},&\text{ if } N=1, \\ [3mm]
  A^{1-\gamma}\sim\dfrac1{L^2}\log R,&\text{ if } N=2, \\ [4mm]
  A\sim A^*(1+cR^{2-N}),&\text{ if } N\ge3.
\end{array}
\end{equation}
And in the case $\gamma=1$ we have, as $L\to L^*$,
\begin{equation}
  \label{eq.R-L}
\begin{array}{ll}
  R\sim\dfrac 2{L},&\text{ if } N=1, \\ [3mm]
  R\sim Le^{\frac2{L^2}},&\text{ if } N=2, \\ [4mm]
  R\sim c(L-L^*)^{-\frac1{N-2}},&\text{ if } N\ge3.
\end{array}
\end{equation}
See the proof of Theorem~\ref{teo-stat} for the values of $A^*,\,L^*,\,L_1$ and $r_0$.

\begin{rem}\label{rem-dirichlet}
We have proved that the Dirichlet problem in a ball corresponding to the equation in \eqref{eq.principal} has stationary solutions  only below the Sobolev exponent $p_S$ if $N\ge3$ or for every $p>0$ if $N\le2$.
The solutions to the Dirichlet problem can be used in comparison arguments as subsolutions to the Cauchy problem.
\end{rem}

\subsection{Exponential solutions for the linear equation}\label{expo-linear}

\

We   look for explicit radial global unbounded solutions in the case $p=m=1$. The length $L$ plays a fundamental role in the existence. We try solutions in the form
$$
u(x,t)=e^{\lambda t}\varphi_\lambda(|x|)
$$
where the profile $\varphi_\lambda$ satisfies two Bessel equations
\begin{equation}\label{two-bessel}
\left\{
\begin{array}{ll}
\varphi''+\frac{N-1}r\varphi'+(1-\lambda)\varphi=0,&\quad\text{if } 0<r<L,\\
\varphi''+\frac{N-1}r\varphi'-\lambda\varphi=0,&\quad\text{if } r>L,\\
\varphi(r)>0,&\quad\text{for } r\ge0,\\
\varphi'(0)=0.
\end{array}
\right.
\end{equation}

\begin{teo}\label{prop-lambda0}
Given any $L> L^*(N)$ there exists a unique value $\lambda_0=\lambda_0(L)\in(0,1)$ for which there exists a  solution $\varphi_{\lambda_0}\in C^1([0,\infty))$ of~\eqref{two-bessel}. The solution is unique up to multiplicative constants.
\end{teo}
\begin{proof} The solution of both Bessel equations in \eqref{two-bessel} give
$$
\varphi_\lambda(r)=r^{-\nu}\left\{
\begin{array}{ll}
J_{\nu}(\sqrt{1-\lambda}\,r), \qquad & \text{if } 0<r< L,\\
BI_{\nu}(\sqrt{\lambda}\,r)+CK_{\nu}(\sqrt{\lambda}\,r), \qquad & \text{if } r>L.
\end{array}\right.
$$
Here $J_\nu$ is the  Bessel function of first kind of order $\nu\equiv\frac{N-2}2$, and $I_\nu,\,K_\nu$ are the modified Bessel functions of order $\nu$, respectively of first and second kind.
For the case $N=1$ we refer to \cite{FerreiradePablo}.
Denote also by $\eta_{\nu,k}$  the $k$--th root of $J_\nu$. The condition $\varphi_\lambda(r)>0$ implies $L\sqrt{1-\lambda}<\eta_{\nu,1}$ and also that no modified Bessel function of first kind appear, so $B=0$. Recall that $K_\nu>0$, $K_\nu'<0$ and $K_\nu\sim z^{-1/2}e^{-z}$ at infinity. Functions with $B\ne0$ will be useful as subsolutions.

The compatibility conditions at $r=L$ gives the value of $C$ and $\lambda$. First, continuity implies
$$
C=\dfrac{J_{\nu}(L\sqrt{1-\lambda})}{K_{\nu}(L\sqrt{\lambda})}>0.
$$
Now differentiability  fixes the value of $\lambda$ in terms of $L$ if there exists a solution to the equation
\begin{equation}\label{lambda-L}
\Phi(\lambda,L)\equiv\sqrt{1-\lambda}\,\dfrac{J_{\nu}'(L\sqrt{1-\lambda})}{J_{\nu}(L\sqrt{1-\lambda})}-
\sqrt{\lambda}\,\dfrac{K_{\nu}'(L\sqrt{\lambda})}{K_{\nu}(L\sqrt{\lambda})}=0.
\end{equation}
We see next that there always exists a root $\lambda_0=\lambda_0(L)$  if $N=2$, but only for $L$ large if $N>2$.

For $N=2$ we have
$$
\Phi(0,L)=\frac{J_0'(L)}{J_0(L)}<0,\qquad \Phi(1,L)=-\frac{K_0'(L)}{K_0(L)}>0.
$$
There exists a solution for every $L>0$, unique if $L$ is small. As $L$ increases multiple roots appear, due to the zeroes of $J_0(L\sqrt{1-\lambda})$, and we choose the biggest root, $\lambda_0\in(1-\frac{\eta_{0,1}^2}{L^2},1)$, in order  to get $\varphi_{\lambda_0}(r)>0$ for every $0\le r\le L$.

When $N>2$ the function $\Phi(\lambda,L)$ satisfies, for $L>0$ small,
$$
\Phi(\lambda,L)\sim\frac{2\nu}L>0\quad\text{for every } 0<\lambda<1.
$$
There exists then no root. On the other hand,
$$
\Phi(0,L)=\frac{J'_{\nu}(L)}{J_{\nu}(L)}+\frac{\nu}L,\qquad \Phi(1,L)=-\frac{K'_{\nu}(L)}{K_{\nu}(L)}>0.
$$
We see that there is a solution $\lambda_0=\lambda_0(L)$ if and only if $L>\overline L^*$, where  $\overline L^*\in(0,\eta_{\nu,1})$ is the first root of $\Phi(0,L)$.
Observe that this value $\overline L^*$  coincides with the value $L^*$ that appeared in the construction  of the stationary solutions, see~\eqref{stationary-bessel}.
Choosing  as before the largest root $\lambda_0$ when multiple roots appear, we obtain a function $\lambda_0=\lambda_0(L)$ for $L>  L^*$, increasing with $\lim\limits_{L\to (L^*)^+}\lambda_0( L)=0$, $\lim\limits_{L\to\infty}\lambda_0(L)=1$. In fact $\lambda_0\sim1- cL^{-2}$ for $L$ large.
\end{proof}

\subsection{Self-similar solutions of the pure diffusion equation}\label{subsec-self-similar}

We study in this subsection the existence of  radial solutions in self-similar form of two special types for the pure diffusion equation $
u_t=\Delta u^m$ for $x\neq0$. We consider fast diffusion $m<1$, but we restrict ourselves to the so called \emph{not too fast} diffusion range, $m>m^*=\frac{(N-2)_+}N$. This solution will be used in comparison arguments in our problem \eqref{eq.principal} to study the grow-up set for different values of the reaction exponent $p$.

We look for solutions $U=U(r,t)$, $r=|x|$, to the equation
$$
u_t=\Delta u^m,\qquad x\neq0,
$$
of the forms
\begin{equation}\label{eq.self-types}
U(r,t)=t^{\alpha} f(r t^{\beta})\qquad \text{ or }\qquad U(r,t)=e^{\alpha t} f(r e^{\beta t}).
\end{equation}
We denote those solutions as of types I and II, respectively. In both cases the profile $f$ verifies the equation
\begin{equation}
  \label{eq.perfil}
(f^m)''+\frac{N-1}\xi (f^m)'=\alpha f+\beta \xi f',\qquad \xi>0,
\end{equation}
where $f'$ denote $df/d\xi$, and the self-similar exponents satisfy the relation
\begin{equation}
  \label{eq.delta}
  \delta\equiv\alpha (1-m)-2\beta\in \{0,1\}.
\end{equation}
In fact we have $\delta=1$ for solutions of type I and $\delta=0$ for solutions of type II.
As we have said, when using those solutions for comparison we will consider each of those types depending on the value of $p$ in problem \eqref{eq.principal}. We now obtain solutions of the ODE \eqref{eq.perfil} for all values of $\delta\ge0$. We refer to \cite{FerreiradePablo} for the case $N=1$.

\begin{lema} Let  $m_*<m<1$, $\alpha>0$ and $\beta>0$ be three positive parameters such that $\delta=\alpha (1-m)-2\beta\ge 0$. Then, there exists a non-negative decreasing  solution $f$ of \eqref{eq.perfil} for $\xi>0$,  such that $\alpha f+\beta\xi f'\ge 0$. Moreover, the behaviour of $f$ is given by
\begin{equation}
  \label{eq.f-0}
f(\xi)\sim
\left\{
\begin{array}{ll}
1,\qquad & \text{if } N\le2,\\
\xi^{-\frac{N-2}m}, & \text{if } N\ge 3,
\end{array}\right.
\qquad \text{as } \xi\sim 0,
\end{equation}
and
\begin{equation}
  \label{eq.f-infty}
f(\xi)\sim
\left\{
\begin{array}{ll}
\xi^{\frac{-2}{1-m}},\quad & \mbox{if }\delta>0,\\
\xi^{\frac{-2}{1-m}} (\log(\xi))^{\frac1{1-m}}, & \mbox{if }\delta=0,
\end{array}\right.
 \qquad \mbox{as }\xi\to\infty.\qquad
\end{equation}
\end{lema}
\begin{proof}
We assume $N\ge2$ and introduce the following variables
$$
X=\frac{\xi f'}{f}, \qquad Y=\frac1{m} \xi^2 f^{1-m}, \qquad \eta=\log\xi.
$$
The resulting system is
$$
\left\{
\begin{array}{l}
\dot{X}=(2-N)X-mX^2+Y(\alpha+\beta X),\\
\dot{Y}=(2+(1-m)X)Y,
\end{array}\right.
$$
where $\dot{X}=dX/d\eta$. We look for non-negative decreasing profiles, so we focus on the second quadrant $X<0\,,\, Y>0$.

We first consider $\delta>0$, in which case the critical points are
$$
A=(0,0), \qquad B=\left(\frac{2-N}{m},0\right), \qquad C=\left(\frac{-2}{1-m},\frac{4-2N(1-m)}{(1-m)\delta}\right).
$$
Notice that since $m>m_*$ the critical point $C$ belongs to the second quadrant. 
%
%
%
Let us define,
$$
\begin{array}{l}
\displaystyle \Gamma_1=\left\{\frac{-2}{1-m}\le X\le \frac{2-N}{m}\,,\, Y=0\right\},\\[3mm]
\displaystyle \Gamma_2=\left\{X=\frac{-2}{1-m}\,,\, 0\le Y\le \frac{4-2N(1-m)}{(1-m)\delta}\right\},\\[3mm]
\displaystyle \Gamma_3=\left\{\frac{-2}{1-m}\le X\le \frac{2-N}{m}\,,\, Y= \frac{mX^2-(2-N)X}{\alpha+\beta X}\right\}.
\end{array}
$$
Note that in $\Gamma_1\cup\Gamma_2$ we have $\dot{X}\le 0$ and $\dot{Y}=0$,
while in $\Gamma_3$  we have $\dot{X}= 0$ and $\dot{Y}>0$. Then, if we look at the orbits backward in time, the region
$$
\Omega=\left\{\frac{-2}{1-m}\le X\le \frac{2-N}{m}\,,\, 0\le Y\le \frac{mX^2-(2-N)X}{\alpha+\beta X}\right\}
$$
is invariant. Even more, in this region it holds $dY/d\eta>0$, so if we look for the orbit passing through a point in either $\Gamma_2$ or $\Gamma_3$
the only possibility is that it comes from the point $A$ if $N=2$ or $B$ if $N\ge3$. Therefore, there exists a separatrix orbit connecting the points $A$ and $C$ if $N=2$ and the points $B$ and $C$ if $N\ge 3$, see Fig. \ref{fig.planofases}.

\begin{figure}[!ht]
\hspace*{-1.5cm}
\includegraphics[scale=.4]{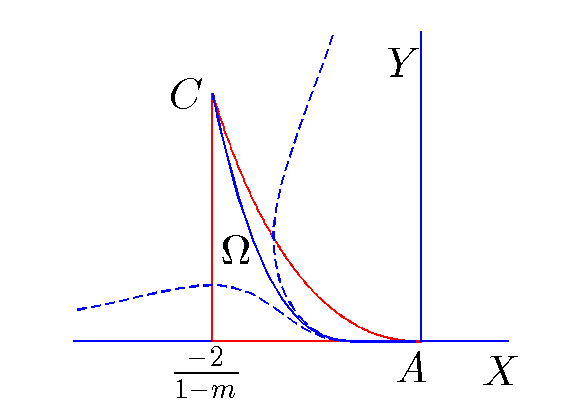}\
\includegraphics[scale=.4]{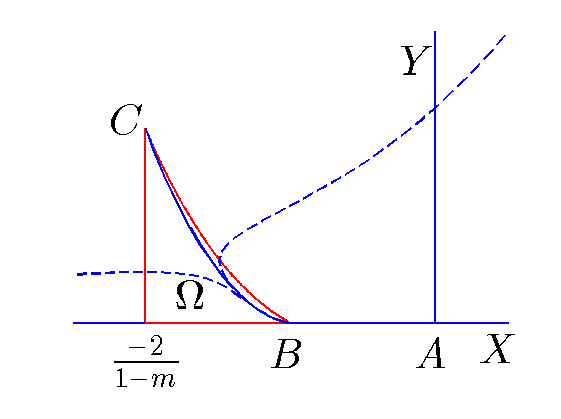}
\caption{The phase plane for $\delta>0$. $N=2$ to the left and $N\ge 3$ to the right.}
\label{fig.planofases}
\end{figure}

This separatrix orbit gives us a decreasing positive trajectory such that
$$
f_*(\xi) \sim \xi^{-\frac2{1-m}} \quad \mbox{as } \xi\to\infty.
$$
For $\xi$ near zero we have  $Y\sim e^{2\eta}$ in dimension $N=2$, while for $N\ge 3$ we have  $X\sim -(N-2)/m$. Therefore
$$
f_*(\xi)\sim
\left\{
\begin{array}{ll}
1,\qquad & \text{if } N=2,\\
\xi^{-\frac{N-2}m}, & \text{if } N\ge 3,
\end{array}\right.
\qquad \text{as } \xi\sim 0.
$$

\

Now we consider $\delta=0$. In this case, the critical point $C$ disappears, but we can use the same argument as before, observing that the separatrix orbit connects the point $A$ with the point $(-2/(1-m),\infty)$ for $N=2$ and the point $B$ with the point $(-2/(1-m),\infty)$ if $N\ge 3$. The picture is the analogous to Fig. \ref{fig.planofases} with the point $C$ going vertically to infinity. We obtain in this way a decreasing positive solution with the same behaviour as before near the origin,  and the behaviour for $\xi$ large
$$
f_*(\xi)\sim \left(\frac{\log\xi}{\xi^2}\right)^{\frac1{1-m}}.
$$

Finally, we observe that in both cases the separatrix orbit lives in $\Omega$ for all the values of the parameter $\xi>0$, which implies  $X\ge -2/(1-m)$. Thus,
$$
\alpha f_*+\beta \xi f_*'= (\alpha+\beta X)f_*\ge \frac\delta{1-m}\,f_*\ge0.
$$
\end{proof}

As a corollary we obtain the existence of grow-up self-similar solutions to the fast diffusion equation.

\begin{teo}\label{cor.ss.pme}
For every $m_*<m<1$ there exists two biparametric families of self-similar solutions $\{U_{\alpha,\mu,I},\,\alpha>1/(1-m),\,\mu>0\}$ and $\{U_{\alpha,\mu,II},\,\alpha>0,\,\mu>0\}$, to the equation $\partial_tu=\Delta u^m$ for $x\neq0$, of types I and II respectively. These solutions are radially decreasing in space and increasing in time.
\end{teo}
\begin{proof} For each $\alpha>1/(1-m)$  in the case of type I, or $\alpha>0$ in the case of type II, we consider the self-similar solution corresponding to the profile $f=f_*$ just constructed with $\beta>0$ satisfying \eqref{eq.delta} and
$$
\lim_{\xi\to0^+}\xi^{\frac{(N-2)_+}m}f(\xi)=1.
$$
Monotonicity in space follows from the property $f'(\xi)\le0$. As to the monotonicity in time we use the fact that $\alpha f(\xi)+\beta\xi f'(\xi)\ge 0$. Now for each $\mu>0$ we consider the self-similar solution with profile $f_\mu(\xi)=\mu^{\frac2{1-m}}f(\mu\xi)$.
\end{proof}

\section{Bounded vs. unbounded  solutions}\label{sect-bdd-notbb}\setcounter{equation}{0}

By the definition of global existence exponent $p_0$, if $p<p_0$ all the solutions are global, while if $p>p_0$ there exist solutions that blow-up in finite time. The value of $p_0$ is given in \eqref{exp-p0}. We study in this section two different questions: $(i)$ if the solutions are global or not in the limit case $p=p_0$; and $(ii)$ if the global solutions for $p\le p_0$ are bounded or not.

Before that we observe that the value of $p_0$ in the case $N\ge2$ and $m<1$ is not covered by the literature, though it is easy to see that $p_0=1$ and that it lies in the global solutions side.
\begin{lema}
  \label{lem-exp} Let $m\le1$. Every solution is global if $p\le1$, while there exist blow-up solutions when $p>1$. Thus $p_0=1$.
\end{lema}
\begin{proof}
  The case $p\le1$  follows by comparison with the supersolution
$$
\overline u(t)=Me^t,\qquad M=\|u_0\|_\infty.
$$

On the other hand, for $p>1$ we can apply Kaplan's method to obtain blow-up solutions if the initial value is large. The method works precisely because $m\le1<p$. To that purpose let $(\lambda_1,\varphi_1)$ be the first eigenvalue and eigenfunction of the Laplacian in the ball $B_L$, normalized such that $\int_{B_L}\varphi=1$. Let $J(t)=\int_{B_L}u\varphi$. We have
  $$
  J'(t)\ge-\lambda_1\int_{B_L}u^m\varphi_1+\int_{B_L}u^p\varphi_1\ge-\lambda_1(\int_{B_L}u\varphi_1)^m+(\int_{B_L}u\varphi_1)^p.
  $$
  If the initial value is large so as to satisfy $\int_{B_L}u_0\varphi>\lambda_1^{\frac1{p-m}}$, then $J'(t)\ge CJ^p(t)$, which means that $J$ (and thus $u$) blows up in finite time.
\end{proof}

Let us now concentrate in the limit case $p=p_0$, and study if the solutions are global or not. By the above lemma we only have to consider the case $m>1$. The unidimensional case is solved in \cite{FerreiradePabloVazquez}, and the solutions are global. The case $N\ge2$ is considered in \cite{Liang}, but the proof of blow-up presented in that paper fails when $N\ge3$ and $L$ is small, precisely by the existence of stationary solutions, see Theorem~\ref{rem-chino}.

\begin{teo}
  \label{teo-bup-p=m} Assume  $p=m$ and let $u$ be the solution to problem~\eqref{eq.principal}.
\begin{enumerate}
  \item If $N=2$ then $u$ blows up in a finite time if $m>1$ and it is global unbounded if  $m\le 1$.
  \item For $N\ge 3$ the behaviour of $u$ depends on $L$:
  \begin{enumerate}
    \item[i)] $u$ is global and bounded if $L\le L^*$, assuming also the behaviour \eqref{extrainfinity} when $L=L^*$;
    \item[ii)] For $L>L^*$ the function $u$ blows up in a finite time if $m>1$ and it is global unbounded if  $m\le 1$.
    \end{enumerate}
\end{enumerate}
Moreover, the global unbounded solutions grow up in some ball of positive radius.
\end{teo}
\begin{proof}
Let $N\ge 3$ (the case $N=2$ is similar to the case $N\ge 3$ and $L>L^*$) and
let us consider the function $g_A=w^{1/m}$ where $w$ es defined in~\eqref{stat-w}.
By Theorem~\ref{teo-stat} we have that for $L<L^*$ taking $A$ large $g_A$ is a supersolution, bigger than $u_0$ at $t=0$ and thus bigger than $u(\cdot, t)$ at any time. This is clear when $L<L^*$, since the stationary solution is strictly positive. If $L=L^*$ we use the behaviour \eqref{extrainfinity}.

On the contrary, when $L>L^*$ the function $g_A$ vanishes at some point $R>0$ independent of $A$. Then, taking $A$ small enough $u$ is a supersolution of the problem
\begin{equation}\label{eq.dirichlet3}
\left\{
\begin{array}{ll}
w_t=\Delta w^m+a(x) w^m,\qquad & |x|<2R, t>0,\\
w(x,t)=0,\qquad & |x|=2R, t>0,\\
w(x,0)=w_0(x),
\end{array}\right.
\end{equation}
where $w_0(x)=g_A(x)$ for $|x|\le R$ and $w_0(x)=0$ in $R\le |x|\le 2R$. Notice that if the initial datum $u_0$ were not positive, by the penetration property of the solutions to the pure diffusion equation without reaction, there exists a time $t_0$ such that the support of $u(\cdot,t_0)$ contains the ball $B_R$ and then taking $A$ small enough $u(x,t_0)\ge w_0(x)$. Thus, again by comparison $u(x,t+t_0)\ge w(x,t)$ for $t\ge 0$.

\

We claim that $w$ is unbounded in $B_R(0)$, and moreover it blows up in a finite time if $m>1$. This then gives that $u$ is global unbounded when $m\le1$ (at least in $B_R(0)$), and blows up if $m>1$.

\

In order to prove the claim we note that problem \eqref{eq.dirichlet3} has no stationary solution, see Theorem~\ref{estacionarias-Dirichlet}. Moreover, since $w_0$ is a radial decreasing function which satisfies $\Delta w_0^m+a(x)w_0^m\ge 0$, we get a radial decreasing solution which is increasing in time.
This implies that the solution can not go to zero and then it must be unbounded. Indeed, let us consider the Lyapunov functional
\begin{equation}\label{eq.lyapunov}
E_w(t)=\frac12 \int_{|x|<2R} |\nabla w^m|^2\,dx-\frac12 \int_{|x|<2R} a(x) w^{2m}\,dx.
\end{equation}
It is nonincreasing,
\begin{equation}\label{eq.lyapunov-decrece}
E_w'(t)=-\frac{4m}{(m+1)^2} \int_{|x|<2R} \left((w^{\frac{m+1}2})_t\right)^2\,dx\le0,
\end{equation}
and also $E_w$ is bounded from below provided $w$ is bounded. Therefore, by standard arguments $w$ converges (up to a subsequence of times) to an stationary solution. Ruled out the possibility to go to the only stationary solution, the trivial one, this implies that $w$ is unbounded, that is,
$$
\limsup_{t\to T} \|w(\cdot,t)\|_\infty =\infty\quad \mbox{for some } T\le\infty
$$
Even more, $w$ is unbounded in $B_R(0)$, since if we suppose that $w(x,t)$ is bounded for  $|x|=R_1<R$, then for $M$ large enough $w$ is a subsolution to
$$
\left\{
\begin{array}{ll}
z_t=\Delta z^m+z^m, \quad & |x|<R_1,\, 0<t<T,\\
z(x,t)=M, & |x|=R_1,\, 0<t<T,\\
z(x,0)=M, & |x|<R_1.
\end{array}\right.
$$
On the other hand, $g_A$ is a stationary supersolution for $A$ large. Then, by comparison $w$ is bounded. This contradiction implies that $w$ is unbounded in $B_R(0)$.

On the other hand, using the concavity argument of  \cite{LS} we obtain that the function
$$
J(t)=\frac{1}{m+1}\int_{|x|<2R} w^{m+1}(x,t)\,dx
$$
satisfies
$$
\begin{array}{rl}
\displaystyle J'(t)&\displaystyle=\int w^m w_t=\int w^m\Delta w^m+\int a(x)w^{2m}=-2E_w(t), \\
\displaystyle(J'(t))^2&\displaystyle=\left(\int w^{\frac{m+1}2}w^{\frac{m-1}2}w_t\right)^2\le\frac4{(m+1)^2}\int w^{m+1}\int \left((w^{\frac{m+1}2})_t\right)^2, \\
\displaystyle J''(t)&\displaystyle=-2E'_w(t)=\frac{8m}{(m+1)^2} \int \left((w^{\frac{m+1}2})_t\right)^2,
\end{array}
$$
and finally
$$
(J'(t))^2\le \frac{m+1}{2m}J(t)J''(t).
$$
Since $J(t)$ is unbounded $J'(t_0)>0$ at some time $t_0$. Moreover $E_w$ is decreasing, thus $J'(t)>0$ for every $t>t_0$. Therefore we can integrate the above inequality to get
\begin{equation}\label{eq.J'}
J'(t)\ge C J^{\frac{2m}{m+1}}(t).
\end{equation}
Let us observe that if $m>1$ the exponent $\frac{2m}{m+1}>1$, then  $J$ (and therefore $w$) blows up in a finite time $T<\infty$.

\end{proof}

This completes the proof of Theorem~\ref{teo-global-p0}.

\begin{rem}\label{rem.tasaBR}
Notice that, since $w$ is radially nonincreasing, inequality \eqref{eq.J'}  for $m<1$ gives the lower estimate $u(0,t)\ge ct^{\frac1{1-m}}$. As we will see this estimate can be extended to the whole $\mathbb{R}^N$, see Lemma \ref{lem-tasa-p=m}.
\end{rem}

\

We now consider exponents $p\le p_0$ and prove Theorem~\ref{teo-GUP}. We start with an easy result.

\begin{teo}\label{teo.m-menor-1}  Let $u$ be a global solution with $p< m$.
Then $u$ is  bounded if and only if $N\ge3$.
\end{teo}
\begin{proof}
If $N\ge3$ we use the fact that there exist large stationary solutions. For $N=2$ we argue by contradiction, assuming $u(x,t)\le K$ for every $x\in\mathbb{R}^2$ and $t>0$.
We consider then the Dirichlet problem
\begin{equation}\label{eq.dirichlet2}
\left\{
\begin{array}{ll}
w_t=\Delta w^m+a(x) w^p,\qquad & |x|<R,\;t>0,\\
w(x,t)=0,\qquad & |x|=R, \;t>0,\\
w(x,0)=w_0(x),& |x|<R,
\end{array}\right.
\end{equation}
for some $R>0$. By \eqref{eq.A-R} we can take $R=R_1$ large such that the corresponding stationary solution $W_{R_1}$ satisfies $W_{R_1}(0)>K$. On the other hand, if we take $R=R_0<R_1$ and $w_0(x)=W_{R_0}(x)$ for $0\le |x|\le R_0$, $w_0(x)=0$ for $R_0\le |x|\le R_1$, we obtain a bounded increasing in time solution $w$ to problem \eqref{eq.dirichlet2}. By standard arguments $\lim_{t\to\infty}w(x,t)=W_{R_1}(x)$.

On the other hand if  $R_0$ is so small in order to have $w_0(x)\le u_0(x)$ for $|x|<R_1$, comparison implies $w(x,t)\le u(x,t)$ for $|x|<R_1$ and every $t>0$. This is a contradiction.
\end{proof}

\begin{rem}\label{rem-pmenorm}
Notice that for $N=2$ and  $p< m$ the solution $u$ must be unbounded in any ball. Indeed, by comparison we can assume that $u$ is radial. Arguing as in the case $p=m$, if  $u$ is bounded on $|x|=R_1$ for some $R_1>0$,  we can put  above $u$ a large stationary solution, so $u$ can not grow-up.
\end{rem}

In order to complete the proof of Theorem \ref{teo-GUP} it only remains to consider the range $m<p\le1$ and show that  there exist unbounded solutions.

\begin{teo}\label{teo.m-menor-1}   If $m<p\le1$   there exist global unbounded solutions to problem \eqref{eq.principal}.
\end{teo}

\begin{proof}
Consider the problem, for some $m<q<p_S$ (and $q\le1$), $R>0$,
\begin{equation}\label{eq.CauchyDirichlet}
\left\{
\begin{array}{ll}
z_t=\Delta z^m+a(x) z^q,\qquad &  |x|<R,\; t>0,\\
z(x,t)=0,&  |x|=R,\; t>0,\\
z(x,0)=z_0(x)&  |x|<R.
\end{array}\right.
\end{equation}
Observe that if $R<L$ the function $\underline z(x,t)=\psi(t)\varphi(x)$ is a subsolution
if $\varphi$ is a stationary solution (see Section~\ref{sect-stationary}) and $\psi$ satisfies
$$
\psi'=\frac{\psi^q-\psi^m}{\varphi^{1-q}(0)},\quad \psi(0)>1.
$$
In fact, since $\psi(t)>1$ for every $t>0$ and $\varphi(x)\le\varphi(0)$ for every $x$, we have
$$
\begin{array}{rl}
\underline z_t-\Delta \underline z^m-a(x) \underline z^q&=\psi'\varphi-(\psi^m-\psi^q)\varphi^q \\ [3mm] &=-\left(1-\left(\dfrac\varphi{\varphi(0)}\right)^{1-q}\right)(\psi^q-\psi^m)\varphi^q\le0
\end{array}$$
for every $|x|<R$, $t>0$.
Since $\psi'\sim\psi^q$, the function $\psi$ tends to infinity and $\underline z$ grows up in $B_R$.

Assume first $m<p<p_S$. Then our solution $u$ is a supersolution to problem~\eqref{eq.CauchyDirichlet} with $q=p$. We conclude grow-up for any initial value above $\underline z(x,0)$. If on the contrary $p_S\le p\le1$, we have that our solution is a supersolution to problem~\eqref{eq.CauchyDirichlet} for any $q<p_S$ provided that $u(x,t)\ge 1$ in $B_L$. Thus, as before, we have grow-up for large initial data. In order to prove that there exists  a solution with $u(x,t)\ge 1$ in $B_L$ we compare with the subsolution $\underline v(x)=\lambda z(x)$, where $z$ is a stationary solution of \eqref{eq.CauchyDirichlet} with some $R>L$ and $\lambda$ is large enough.
\end{proof}

\begin{rem}
  \label{rem-vanish}
  It is well known that if $0<m<\frac{N-2}N$ with $N\ge3$ there exist solutions to the very fast diffusion equation
  $$
  v_\tau=\Delta v^m
  $$
that vanish identically at a finite time $\tau_0$, which depends on the initial value, see for instance \cite{V2}. Take now $v(\cdot,0)$ be such that $\tau_0<1/(1-m)$. Then
$$
w(x,t)=e^{t}v(x,\tau),\quad \tau=\frac{1-e^{-(1-m)t}}{1-m},
$$
is a supersolution to our problem with $p=1$ and it satisfies
$$
w(x,t_0)\equiv0 \quad\text{for every } x\in\mathbb{R}^N,
$$
where $t_0=\frac1{1-m}\log(1-(1-m)\tau_0)$. Therefore, in the case of linear reaction and superfast diffusion in problem \eqref{eq.principal}, any initial value $u_0\le v(\cdot,0)$ produces a solution with finite time extinction.
\end{rem}

\section{Grow-up set}\label{sect-gupset}\setcounter{equation}{0}

The main objective of this section is to study if the unbounded global solutions to problem~\eqref{eq.principal} tend to infinity  for every point $x\in\mathbb{R}^N$. We first remind that the case $p<m$ (which implies $N=2$) follows directly from Remark \ref{rem-pmenorm}, since we can put below the solution a subsolution with grow-up set as large as we want. We therefore deal here with the upper range $m\le p\le1$.

If $p>m$ we also assume that the initial value is a radial function and so it is the solution. We  denote the solution $u(r,t)$, $r>0$, $t>0$, since no confusion arises. We impose the additional condition in that case
\begin{equation}\label{eq.tasah}
\lim_{t\to\infty} u(R,t)=\infty,\quad \mbox{for some } R\in[0, L].
\end{equation}
Notice that if $u$ is bounded in $B_L$ (the region where the reaction takes place), then $u$ is bounded.

Next we prove that under the previous hypotheses the grow-up set is the whole $\mathbb{R}^N$, that is, we prove Theorem \ref{teo-sets}. We divide the proof into several lemmas.

\begin{teo}\label{teo-0-eps}
Assume $R>0$ in \eqref{eq.tasah}. Then $\lim\limits_{t\to\infty}u(r,t)=\infty$ for every $r\ge R$.
\end{teo}
\begin{proof} From \eqref{eq.tasah}, given any $K>0$, there exists $t_K$ a time such that $u(R,t)\ge K$ for all $t\ge t_K$. Now, for any $R_1>R$ we consider the problem
$$
\left\{ \begin{array}{ll}
w_t=\Delta w^m\equiv r^{1-N} (r^{N-1}(w^m)_r)_r ,\quad& r\in (R,2 R_1), \;t>t_K,\\
w(R,t)=K , &  t>t_K,\\
w(2 R_1,t)=0 ,&  t>t_K,\\
w(r,t_K)=w_0(r), & r\in(R,2R_1).
\end{array}\right.
$$
Taking $w_0(r)\le \min\{u(r,t_K),K\}$ a continuous function which satisfies the boundary condition we get that $u$ is a supersolution, then $u\ge w$ for $t>t_K$. It is easy to see that any solution to this problem converges as $t\to\infty$ to the explicit stationary solution
$$
h(x)= K \left( \frac{(2R_1)^{N-2}-r^{N-2}}{(2R_1)^{N-2}-R^{N-2}}\right).
$$
Thus taking $t$ large enough we can get $u(R_1,t)\ge cK$, which is as large as we please.
\end{proof}

\begin{teo}\label{teo.global.p<1}
Let $p\le 1$  and assume $0<R<L$ in \eqref{eq.tasah}. Then $\lim\limits_{t\to\infty}u(r,t)=\infty$ for every $r\le R$.
\end{teo}
\begin{proof} As before, for every $K>0$ we consider a time $t_K$ such that $u$ is a supersolution to the problem in the ball
$$
\left\{ \begin{array}{ll}
w_t=\Delta w^m+ w^p, \quad&  0<r<R,\; t>t_K,\\
w(R,t)=K,  &  t>t_K,\\
w(r,t_K)=u(r,t_K) ,& 0<r<R.
\end{array}\right.
$$
Notice that the flat solution, that is, the solution of $W'=W^p$ with initial datum $W(0)=\inf\limits_{x\in B_R(0)} u(\cdot,t_K)$, is a subsolution for $t\in(T_K,T)$, where $W(T)=K$. Therefore by comparison $u(r,T)\ge K$ for $0\le r\le R$.
\end{proof}

We finally consider the case $R=0$ in \eqref{eq.tasah}.

\begin{teo}\label{teo.globalfast}
Let $0<m<p\le 1$, and assume $\lim\limits_{t\to\infty} u(0,t)=\infty$. Then $\lim\limits_{t\to\infty}u(r,t)=\infty$ for every $r>0$.
\end{teo}
\begin{proof}
We use the Intersection Comparison technique with respect to the family
of positive radial stationary solutions
$$
\left\{\begin{array}{ll}
\Delta w^m+w^p=0, \qquad & 0<r<R.\\
w(R)=0,\;
(w^m)'(0)=0,
\end{array}\right.
$$
which are constructed in Subsection \ref{sect-stationary}. That is, we study the number of sign changes between $w$ and $u$ in $[0,R]$,
$$
N(t)=N(w(\cdot,t),u(\cdot,t)).
$$
The main result of the Intersection Comparison argument asserts that the number of
sign changes between two solutions of a large class of nonlinear parabolic equations,
which includes our equation, does not increase in time provided no new intersections appear through the boundary, cf. \cite{Galaktionov},  \cite{Matano}.

First, since $p>m$ there exists $R_0$ small enough such that for $R\le R_0$ the initial intersection number is $N(0)=1$. Also, as $u(R,t)>0$ no new intersection can appear through $r=R$. Now, at $r=0$ we define $t_0$ and $R<R_0$ such that $u(0,t_0)=w(0)$ with $u_t(0,t_0)>0$ (notice that $t_0$ and $R$ exist because $u(0,t)\to \infty$) then
$$
\begin{array}{rl}
0<u_t(0,t_0)&\displaystyle=\Delta u^m(0,t_0)-\Delta w^m(0,t_0)+u^p(0,t_0)-w^p(0,t_0) \\ [3mm]
&\displaystyle=\Delta( u^m(0,t_0)-\Delta w^m(0,t_0)).
\end{array}
$$
 This implies that the function $z(r)=u^m(r,t_0)-w^m(r)$ satisfies $z(0)=z'(0)=0$, $z''(0)>0$, so that $z(r)>0$  for $r\in(0,\delta)$ and some $\delta>0$. So not only no new intersections appear through $r=0$ but an intersection is lost at that point at time $t=t_0$, i.e. $N(t_0)=0$. By intersection comparison this implies  $N(t)=0$ for $t\ge t_0$.

Therefore $u(r,t_1)> w(r)$ for some $t_1>t_0$. Applying the same argument given in the proof of Theorem \ref{teo.m-menor-1} we get that $u$ grows up in $B_R(0)$. Then, by Theorem \ref{teo-0-eps} $u$ has global grow-up.
\end{proof}

\begin{proof}[Proof of Theorem~\ref{teo-sets}]
As we have said, if $p<m$ the global grow-up follows from Remark \ref{rem-pmenorm}.
In the case $p=m$ we first use Theorem \ref{teo-bup-p=m} to get that the solutions is unbounded at every point in some ball, and the by Theorem~\ref{teo-0-eps} the same holds in  every ball.
Finally if $m<p\le 1$ the result follows by applying Theorems \ref{teo-0-eps}, \ref{teo.global.p<1} and \ref{teo.globalfast}.
\end{proof}

\section{Grow-up rate}\label{sect-guprate}\setcounter{equation}{0}

This section is devoted to  study  the speed at which the global unbounded solutions to problem~\eqref{eq.principal} tend to infinity.

An easy upper estimate of the grow-up rate for $0<p\le1$ is given by the solutions of the ODE
$$
U'(t)=U^p(t).
$$
This gives,
\begin{equation}
  \label{flat}
  u(x,t)\le\begin{cases}  (M^{1-p}+(1-p)t)^{\frac1{1-p}},& \quad \text{if } p<1, \\
  Me^t,& \quad \text{if } p=1,
  \end{cases}
\end{equation}
where $M=\|u_0\|_\infty$. We have named this bound the natural rate.

We see next that these estimates are far from being sharp in some cases, and indeed, in the cases when the solution grows up with the natural rate, it does so only in the ball $B_L$, where the reaction applies.

\subsection{Linear diffusion, sublinear reaction, $m=1$, $p<1$}

\

We prove here Theorem \ref{teo-rates-p<1} by means of  Duhamel's formula,
\begin{equation}\label{duhamel}
u(x,t)=\int_{\mathbb{R}^N}u_0(y)\Gamma(x-y,t)+ \int_0^t\int_{|y|\le L} u^p(y,s)\Gamma(x-y,t-s)\,dy\,ds,
\end{equation}
where $\Gamma$ is the Gauss kernel, and we consider every dimension $N\ge1$ for completeness. The formal proof is as follows: if $u(x,t)\sim g(t)$ in $B_L$ for $t\ge1$, then
$$
\begin{array}{rl}
g(t)&\displaystyle\sim\int_1^t\int_{|y|\le L} g^p(s)\Gamma(x-y,t-s)\,dy\,ds \sim\int_1^t g^p(s)
\int_0^{\frac{L^2}{4(t-s)}} r^{\frac{N-2}2}e^{-r}\,dr\,ds\\ [3mm]
&\displaystyle\sim\int_1^t g^p(s)s^{-\frac N2}\,ds.
\end{array}
$$
Now, the solution of the resulting differential equation $g'(t)\sim g^p(t)t^{-\frac N2}$ is
$$
g^{1-p}(t)\sim\begin{cases} 1+t^{\frac{2-N}2},& \text{ if } N\ne2, \\ \log t,& \text{ if } N=2.\end{cases}
$$
The case $N=1$ was obtained in \cite{FerreiradePablo}. In the case $N\ge3$ this also explains why the solution must be bounded.

We then proceed with the detailed proof

\begin{proof}[Proof of Theorem \ref{teo-rates-p<1}]
Assume $u(x,t)\le g(t)$ for every $|x|\le L$ and some increasing function $g$. Since the first term in \eqref{duhamel} is bounded, we have
$$
\begin{array}{rl}
u(x,t)&\displaystyle\le 2\int_1^t\int_{|x-y|\le 2L} g^p(s)(4\pi(t-s))^{-1}e^{-\frac{|x-y|^2}{4(t-s)}}\,dy\,ds\\ [3mm]
&\displaystyle
\le c\int_1^t g^p(s) \int_0^{\frac{L^2}{t-s}} e^{-r}\,dr\,ds\le  c g^p(t)\int_1^t(1-e^{-\frac{L^2}{t-s}})\,ds\\ [4mm]
&\displaystyle\le cL^2 g^p(t)\log t.
\end{array}
$$
In the last step we have used L'H\^opital's rule,
$$
\lim_{t\to\infty}\frac{\displaystyle\int_1^t (1-e^{-\frac{L^2}{t-s}})\,ds}{\log t}=\lim_{t\to\infty}\frac{\displaystyle\int_{\frac{L^2}{t-1}}^\infty (1-e^{-z})z^{-2}\,dz}{\log t}=L^2.
$$
We iterate this estimate starting with $g_1(t)=\nu t^{\frac1{1-p}}$, see~\eqref{flat}.
We obtain in this way the sequence
$$
g_k(t)= c_kt^{\delta_k}(\log t)^{\sigma_k}, \qquad \delta_{k+1}=p\delta_k,\quad \sigma_{k+1}=p\sigma_k+1,\quad c_{k+1}=cL^2c_k^p.
$$
We end with the limits,
$$\lim_{k\to\infty}\delta_k=0,\quad\lim_{k\to\infty}\sigma_k=\dfrac1{1-p},\quad\lim_{k\to\infty}c_k= (cL^2)^{\frac1{1-p}}.$$

As to the lower estimate, it is clear first that $u(x,t)\ge C_0>0$ in $B_L$  for $t>0$, since we can put below a small stationary subsolution. Assume that  we have $u(x,t)\ge g(t)$ for every $|x|<L$, $t>0$. Then the above Duhamel's formula gives
$$
u(x,t)
\ge\displaystyle c \int_1^t g^p(s)
\int_0^{\frac{(L-|x|)^2}{4(t-s)}} e^{-r}\,dr\,ds=\int_1^t g(s)(1-e^{-\frac{L^2}{t-s}})\,ds.
$$
We now observe that
for every $q\ge0$
$$
\int_1^t (\log s)^q(1-e^{-\frac{c}{t-s}})\,ds\ge c(\log t)^{q+1},
$$
to get, again by iteration,
$$
u(x,t)\ge c(L-|x|)^{\frac1{1-p}}(\log t)^{\frac1{1-p}}.
$$

We have just proved, for instance for $|x|<L/2$, that $u(x,t)\ge h(t)=c(\log t)^{\frac1{1-p}}$. We now extend this estimate to every compact of $\mathbb{R}^2$.

To this purpose we use that $u$ is a supersolution to the problem
$$
\left\{
\begin{array}{ll}
w_t=\Delta w,\qquad & |x|>L/2,\; t>t_0,\\
w(x,t)=h(t), &|x|=L/2,\; t>t_0,\\
w(x,t_0)=u(x,t_0), & |x|>L/2,
\end{array}\right.
$$
and a subsolution can be found explicitely. In fact, for $R>L$ fixed the function
$\underline w(x,t)=h(t)\varphi(|x|)^\gamma$, where $\varphi= (1-|x|/R)_+$ do the job provided that both $t_0$ and $\gamma$ are large enough.
Indeed,
$$
\underline w_t-\Delta\underline w=h\varphi^{\gamma-2}\left(\frac{h'}{h}\varphi^2+
\frac{\gamma}{R} \Big(\frac{1}{r}-\frac{\gamma}{R}\Big)\right).
$$
Since $h'/h\to 0$ as $t\to\infty$ we can take $t_0$ such that $\varphi^2 h'/h\le 1$ for all $t\ge t_0$. Then for  $\gamma$ large enough $\underline w_t-\Delta\underline w\le 0$.

Moreover, $\varphi(r)<(1-L/R)_+<1$, then $\underline w(x,t)<h(t)$ and
taking $\gamma$ large enough $\underline w(x,t_0)< u(x,t_0)$.

This means that given any $x_0\in\mathbb{R}^2$ we can define $R=2|x_0|>L$ to have
$$
u(x_0,t)\ge \underline u(x_0,t)=h(t) 2^{-\gamma}.
$$

\end{proof}

\subsection{Linear diffusion, linear reaction, $m=p=1$}

\

In this case we also have that the presence of a localized reaction provokes a growth of the solutions that is strictly slower than that of the solutions with global reaction, that is, we prove that the solutions behave for large times like an exponential, but the exponent depends on the length $L$ and is strictly less than 1.
We  use the explicit radial global unbounded solutions obtained in Subsection \ref{expo-linear} in order  to establish an estimate of the growth of general solutions.

\begin{proof}[Proof of Theorem \ref{teo-rates-N>1}]
Let $u$ be a solution to problem~\eqref{eq.principal} where $L>L^*$. By comparison from above and below with the solutions obtained in Theorem~\ref{prop-lambda0} with different values of $\lambda$, we get that there exists a  function $c(s)$, decreasing with $\lim\limits_{s\to0}c_1(s)=0$, such that
$$
 c(\varepsilon)e^{(\lambda_0-\varepsilon)t}\le u(x,t)\le c^{-1}(\varepsilon) e^{(\lambda_0+\varepsilon)t},
$$
for every  $x\in\mathbb{R}^N$, $t\ge1$, $\varepsilon>0$. We conclude \eqref{rates-p=1}.
\end{proof}

\subsection{The critical line $m=p<1$}

\

In this parameter, we show that the natural grow-up rates \eqref{flat}  are sharp by proving the  lower bound.

\begin{lema}\label{lem-tasa-p=m}
Let $p=m<1$ and let $u$ be a global unbounded solution of \eqref{eq.principal}. Then
$$
u(\cdot,t)\ge ct^{\frac1{1-m}}
$$
uniformly in compact sets of $\mathbb{R}^N$.
\end{lema}
\begin{proof}
We consider a solution in separated variables, $w(x,t)=\psi(t)\varphi(|x|)$, where $\psi=\psi_\lambda$ satisfies $\psi'=\lambda\psi^m$, and $\varphi=\varphi_\lambda$ is a solution to
$$
\left\{
\begin{array}{ll}
( \varphi^m)''+\frac{N-1}r(\varphi^m)'+a(r)\varphi^m-\lambda \varphi=0, \quad & r>0,\\
\varphi(0)=1,\ (\varphi^m)'(0)=0.
\end{array}\right.
$$
It is easy to check that if $L>L^*$ there exists a limit value $\lambda^*>0$ such that for every $0<\lambda<\lambda^*$ the solution $\varphi$ is positive and decreasing in $[0,R_\lambda)$, with $\varphi(R_\lambda)=0$, and $\lim_{\lambda\to\lambda^*}R_\lambda=\infty$.
In fact, by the results in Section~\ref{sect-stationary} we know that $\varphi$ crosses the axis at some point if $\lambda=0$, so by continuous dependence with respect to $\lambda$ the same holds  when $\lambda$ is small.
On the other hand, the solution corresponding to $\lambda=1$ satisfies  $\varphi(r)=1$ in $0<r<L$ and it increases to infinity for $r>1$. The existence of $\lambda^*$ is now standard.

Let now $x_0\in\mathbb{R}^N$ be any point. We take $\lambda\sim\lambda^*$ so that $R_\lambda>|x_0|$. Comparison in $[0,R_\lambda]$ gives the grow-up rate in the ball $B_{|x_0|}$.
\end{proof}

\subsection{The supercritical case $m< p\le 1$}

\

As in the previous case we show here that the grow-up rate of our solutions is the natural one, but only inside the ball $B_L$, where the reaction takes place.

\begin{lema}\label{lem.bola}
Let $m<p\le1$, and also $p<p_S$ if $N\ge3$. Let $u$ be a solution of \eqref{eq.principal} with global grow-up. Then for every $|x|<L$ it holds
$$
u(x,t)\ge c \left\{
\begin{array}{ll}
t^{\frac1{1-p}},\quad &\text{if}\quad p<1,\\
e^t, & \text{if}\quad p=1.
\end{array}\right.
$$
\end{lema}
\begin{proof}
We compare with the subsolution in separated variables $\underline z$ given in the proof of Theorem \ref{teo.m-menor-1} with $R=L$. Let $t_1>0$ be such that $u(x,t_1)\ge\underline z(x,0)$ for $|x|<L$. We obtain  $u(x,t)\ge \phi(x)\psi(t-t_1)$, and conclude with the behaviour of $\psi'\sim\psi^p$ as $t\to\infty$.
\end{proof}
 In the case $p\ge p_S$ we only can compare with a subsolution satisfying $\psi'\sim\psi^q$, $q<p_S$, thus obtaining $t^{\frac1{1-q}}$ as grow-up rate, which is presumed not to be sharp.

The next task is to obtain the grow-up rate outside the ball $B_L$. We show that there the rate is strictly smaller. To that purpose we consider the self-similar solutions constructed in Section \ref{sect-special}.
\begin{proof}[Proof of Theorem \ref{teo.tasas.fuera}] Let us consider first the case $p<1$.
By Lemma~\ref{lem.bola} we know that there exits a constant $C_1>0$ and a time $t_0>0$ such that $u(x,t)\ge C_1 t^{\frac1{1-p}}$ for $|x|=L/2$, $t\ge t_0$, so $u$ is a supersolution to the problem
$$
\left\{
\begin{array}{ll}
w_t=\Delta w^m, \qquad &|x|>L/2,\; t>t_0,\\
w(x,t)=C_1 t^{\frac1{1-p}}, &|x|=L/2,\; t>t_0,\\
w(x,t_0)=u_0(x).
\end{array}\right.
$$
Put $\underline w(x,t)=AU(|x|,t-t_0)$, where $U=U_{\frac1{1-p},1,I}$ is the selfsimlar solution of type I given in Theorem \ref{cor.ss.pme} with $\mu=1$ and $\alpha=1/(1-p)$, which implies $\beta=\alpha(p-m)/2$.
Since $\underline w_t\ge 0$ we get
$$
\underline w_t-\Delta \underline w^m=(A-A^m)\underline w_t\le 0
$$
provided $A<1$. Let us look at the initial time $t=t_0$. If $N=2$ the profile $f$ defining $U$ is bounded, which implies $\underline w(x,t_0)\equiv0$. When $N\ge3$, using the fact that $p< p_S$ implies that $\alpha-\beta(N-2)/m>0$,  the behaviour near the origin of $f$, see \eqref{eq.f-0}, gives us
$$
\underline w(x,t_0)\sim A (t-t_0)^{\alpha-\beta\frac{N-2}m}|x|^{-\frac{N-2}m} \to 0 \quad \text{for } |x|\ge L/2,\;\text{as } t\to t_0.
$$
On the other hand, we note that for $|x|=L/2$ it holds
$$
\lim_{t\to t_0}\frac{(t-t_0)^\alpha f_*((t-t_0)^\beta L/2)}{t^\alpha}=0=
\lim_{t\to \infty}\frac{(t-t_0)^\alpha f_*((t-t_0)^\beta L/2)}{t^\alpha}.
$$
Then there exists $A>0$ small such that
$$
\underline w(x,t)=A (t-t_0)^\alpha f_*(\frac{L}2(t-t_0)^\beta)<C_1 t^{\alpha} \qquad \text{for } |x|=L/2,\; t\ge t_0.
$$
Then by comparison $u\ge \underline w$, and thus for $t$ large and $|x|>L$ it holds, using \eqref{eq.f-infty},
$$
u(x,t)\ge \underline w(x,t)\sim |x|^{\frac{-2}{1-m}} t^{\frac{1}{1-m}}.
$$

In order to obtain the upper estimate, we observe that from \eqref{flat} we have $u(L,t)\le C_2 (t+1)^{\frac1{1-p}}$ for $t\ge 0$. Then, $u$ is a subsolution of
$$
\left\{
\begin{array}{ll}
w_t=\Delta w^m, \qquad &|x|>L,\; t>0,\\
w(x,t)=C_2(t+1)^{\frac1{1-p}},&|x|=L,\; t>0,\\
w(x,t_0)\ge u(x,t_0).
\end{array}\right.
$$
Here we consider the function $\overline w(x,t)=AU(|x|-L,t+1)$, where $U=\widetilde U_{\frac1{1-p},1,I}$ is the one dimensional self-similar solution  (that is, $ U_t=( U^m)_{xx}$, $x>0$), with the same exponents as before. First observe that  since the profile satisfies $f'\le0$, it is a supersolution to our multidimensional equation. Using now  the behaviour at infinity of $u_0$ it is easy to see that for $A$ large enough $\overline w$ is a supersolution to the above problem. Then by comparison we get, for $t$ large and $|x|>L$,
$$
u(x,t)\le \overline w(x,t)\sim |x|^{\frac{-2}{1-m}} t^{\frac{1}{1-m}}.
$$

The case $p=1$ follows in a similar way using this time comparison with a self-similar solution $U=U_{1,1,II}$, that is a solution of type II with $\alpha=1$, $\beta=(1-m)/2$. We  only have to take into account that:

i) By Lemma \ref{lem.bola} we have $u\sim e^t$ for $|x|< L$ and $t>t_0$. On the other hand, $u$ is a supersolution of the fast diffusion equation, so $u(x,t)\sim 1 $ for $|x|< L$ and $t\in[0,t_0]$. Therefore, $u\sim e^t$ for $|x|< L$ and $t>0$. Thus, we can repeat the same argument as before with $t_0=0$.

ii) The comparison at time $t=0$ follows thanks to the behaviour imposed to $u_0$.

Observe finally that for $t\to\infty$ we have, thanks to \eqref{eq.f-infty},
$$
U(x,t)=e^tf(|x|e^{\frac{m-1}2\,t})\sim |x|^{-\frac2{1-m}}(\log|x|)^{\frac1{1-m}}\,t^{\frac1{1-m}}.
$$
\end{proof}

\section*{Acknowledgments}

Work supported by the Spanish project  MTM2014-53037-P.


\begin{thebibliography}{B}

\bibitem{AguirreEscobedo}
{\sc J. Aguirre and M. Escobedo.}
{\it A Cauchy problem for $u_t-\Delta u=u^p$ with $0<p<1$. Asymptotic behaviour
of solutions.}  Ann. Fac. Sci. Toulouse Math. {\bf8}  (1986), 175--203.

\bibitem{BaiZhouZheng}
{\sc X. Bai, S. Zhou and  S. Zheng.}
{\it Cauchy problem for fast diffusion equation with localized reaction.} Nonlinear Anal.
{\bf74} (2011), 2508--2514.

\bibitem{FerreiradePablo}
{\sc R. Ferreira, A. de Pablo.}
{\it Grow-up for a quasilinear heat equation with a localized reaction.}
 Preprint.

\bibitem{FerreiradePabloPerezRossi}
{\sc R. Ferreira, A. de Pablo, M. P\'erez-Llanos and J. D. Rossi.}
{\it Critical exponents for a semilinear parabolic equation with variable reaction.}
 Proc. Roy. Soc. Edinburgh Sect. A {\bf142} (2012), 1027--1042.

\bibitem{FerreiradePabloVazquez}
{\sc R. Ferreira, A. de Pablo and J.L. V\'{a}zquez}. {\it Blow-up
for the porous medium equation with a localized reaction}. J. Differential Equations  {\bf231} (2006),  195--211

\bibitem{Fujita}
{\sc H. Fujita}.
{\it On the blowing-up of solutions of the Cauchy
problem for $u_t=\Delta u+u^{1+\alpha}$}, J. Fac. Sci. Univ. Tokyo
Sec. IA Math. {\bf 16} (1966), 105--113.

\bibitem{Galaktionov} {\sc V.A. Galaktionov }
Geometric Sturmian theory of nonlinear parabolic equations and applications.
Applied Mathematics and Nonlinear Science Series, 3. Chapman \& Hall/CRC, 2004


\bibitem{GalaktionovKurdyumovMikhailovSamarski}
{\sc V.A. Galaktionov, S.P. Kurdyumov, A.P. Mikhailov and  A.A.
Samarski\u\i}. {\it Unbounded solutions of semilinear parabolic
equations}, Keldysh Ins. Appl. Math. Acad. Sci. USSR, Preprint No.
161 (1979).

\bibitem{GalaktionovVazquez}
{\sc V.A. Galaktionov and J.L. Vazquez} {\it A stability technique for evolution partial differential equations. A dynamical systems approach.} Progress in Nonlinear Differential Equations and their Applications, 56. Birkhäuser Boston, Inc., Boston, MA, 2004.

\bibitem{Hayakawa}
{\sc K. Hayakawa}.
{\it On nonexistence of global solutions of
some semilinear parabolic differential equations}, Proc. Japan
Acad. {\bf 49} (1973), 503--505.

\bibitem{herrero-pierre}
{\sc M. Herrero and M. Pierre}
{\it The Cauchy problem for $u_t=\Delta u^m$ when $0<m<1$},
Trans. Amer. Math. Soc. {\bf 291} (1985), 145--158.

\bibitem{Matano} {\sc H. Matano}, {\it Nonincrease of the lap number for a one-dimensional semilinear parabolic equation},
J. Fac. Sci. Univ. Tokyo, Sect. IA {\bf 29} (1982) 401--440.

\bibitem{LS}
{\sc H.A. Levine and P. Sacks}.
{\it Some existence and nonexistence theorems for solutions of degenerate parabolic equations}
J. Differential Equations {\bf 52} (1984) 135--161.


\bibitem{Liang}
{\sc Z. Liang}.
{\it On the critical exponents for porous medium equation with a localized reaction in high dimensions}. Comm. Pure Appl. Anal. {\bf11} (2012),  649--658.

\bibitem{dePabloVazquez}
{\sc A. de Pablo and J.~L. V\'azquez}. {\it The balance between strong reaction and slow diffusion}.  Comm. Partial Differential Equations {\bf15} (1990),  159--183.

\bibitem{dePabloVazquez2}
{\sc A. de Pablo and J.~L. V\'azquez}. {\it Travelling waves and finite propagation in a reaction-diffusion equation}. J. Differential Equations 93 (1991), no. 1, 19--61.

\bibitem{dePabloVazquez3}
{\sc A. de Pablo and J.~L. V\'azquez}. {\it An overdetermined initial and boundary-value problem for a reaction-diffusion equation}. Nonlinear Anal. 19 (1992), no. 3, 259--269.

\bibitem{Pinsky}
{\sc R.G. Pinsky}. {\it Existence and nonexistence of global solutions for $u_t = \Delta u + a(x)u^p$ in $R^d$}.  J. Differential Equations {\bf133} (1997), 152--177

\bibitem{PolacikYanagida}
{\sc P. Pol\'a\v{c}ik and E. Yanagida}. {\it On bounded and unbounded global solutions of a supercritical semilinear heat equation}. Math. Ann. {\bf327} (2003),  745--771.

\bibitem{PolacikYanagida2}
{\sc P. Pol\'a\v{c}ik and E. Yanagida}. {\it  Global unbounded solutions of the Fujita equation in the intermediate range}. Math. Ann. {\bf360} (2014), 255--266.

\bibitem{PolubarinovaKochina} {\sc P.Ya. Polubarinova-Kochina}. {\it On a nonlinear differential equation
encountered in the theory of infiltration}. Dokl. Akad. Nauk SSSR {\bf63} (1948) 623--627.

\bibitem{QV}
{\sc F. Quir\'os and J.~L. V\'azquez}. {\it Asymptotic behaviour of the porous media equation in an exterior domain}. Ann. Scuola Norm. Sup. Pisa Cl. Sci. (4) {\bf28} (1999), 183--227.

\bibitem{Quittner}
{\sc P. Quittner}. {\it A priori bounds for global solutions of a semilinear parabolic problem}. Acta Math. Univ. Comenian. (N.S.) {\bf68} (1999), 195--203.



\bibitem{QuittnerSouplet}
{\sc P. Quittner and Ph. Souplet}. Superlinear parabolic problems. Blow-up, global existence and steady states. Birkh\"auser Advanced Texts.  Birkh\"auser Verlag, Basel, 2007.


\bibitem{SamarskiGalaktionovKurdyumovMikhailov}
{\sc A. Samarski, V. A. Galaktionov, S. P. Kurdyumov and A. P. Mikailov.}
Blow-up in quasilinear parabolic equations. Walter de Gruyter, Berlin, 1995.

\bibitem{V2}  {\sc J.~L. V{\'a}zquez.} Smoothing and decay estimates
for nonlinear diffusion equations. Equations of porous medium
type. Oxford Lecture Series in Mathematics and its Applications,
33. Oxford University Press, Oxford, 2006.


\bibitem{V1}  {\sc J.~L. V{\'a}zquez.} The porous medium equation. Mathematical theory. Oxford Mathematical Monographs. Oxford University Press, Oxford, 2007.


\end{thebibliography}
\end{document}